\documentclass[11pt]{amsart}
\usepackage[T1]{fontenc}
\numberwithin{equation}{section}

\usepackage{float}
\usepackage{tikz-cd}
\usepackage{amsmath}%
\usepackage{amsfonts}%
\usepackage{amsthm}
\usepackage{amssymb}%
\usepackage{graphicx}
\usepackage{hyperref}

\usepackage{float}
\usepackage{tikz-cd}
\usepackage{amsmath}%
\usepackage{amsfonts}%
\usepackage{amsthm}
\usepackage{amssymb}%
\usepackage{graphicx}
\usepackage{hyperref}
\usepackage{enumerate}
\usepackage{comment}

\usepackage{geometry}
\usepackage{layout}
\usepackage[utf8]{inputenc}
\usepackage{amsmath}
\usepackage{graphicx}
\usepackage{amsmath}
\usepackage{amssymb}
\usepackage{amsfonts}
\usepackage{graphicx}
\usepackage{url}
\usepackage{MnSymbol}
\usepackage{arydshln}
\usepackage{outlines}
\usepackage{tikz-cd}

\usepackage{subfig}
\usepackage{amsthm}

\newcommand{\Z}{\mathbb{Z}}

\newcommand{\cA}{\mathcal{A}}
\newcommand{\cD}{\mathcal{D}}
\newcommand{\cB}{\mathcal{B}}
\newcommand{\cH}{\mathcal{H}}
\newcommand{\cM}{\mathcal{M}}
\newcommand{\cS}{\mathcal{S}}
\newcommand{\cJ}{\mathcal{J}}
\newcommand{\F}{\mathbb{F}}

\renewcommand{\O}{\mathcal{O}}
\renewcommand{\epsilon}{\varepsilon}

\renewcommand{\P}{\mathbb{P}}

\newcommand{\cMbar}{\overline{\mathcal{M}}}
\newcommand{\cAbar}{\tilde{\mathcal{A}}}

\newcommand{\dirlim}{\varinjlim}

\renewcommand{\o}{\varnothing}

\renewcommand{\o}{\varnothing}

\newcommand{\rank}{\mathrm{rank}}

\newcommand{\prank}{p\text{-}\rank}

\usepackage{multirow}

\newcommand{\<}{\left \langle}
\renewcommand{\>}{\right \rangle}

\usepackage{xcolor}

\makeatletter
\newtheorem*{rep@thm}{\rep@title}
\newcommand{\newrepthm}[2]{%
\newenvironment{rep#1}[1]{%
 \def\rep@title{#2 \ref{##1}}%
 \begin{rep@thm}}%
 {\end{rep@thm}}}
\makeatother

\newtheorem{thm}{Theorem}[section]

\newtheorem{lem}[thm]{Lemma}
\newtheorem{prop}[thm]{Proposition}
\newtheorem{cor}[thm]{Corollary}
\newtheorem{conj}[thm]{Conjecture}
\newtheorem{prob}[thm]{Open problem}

\newrepthm{thm}{Theorem}

\theoremstyle{definition}
\newtheorem{dfn}[thm]{Definition}

\theoremstyle{remark}
\newtheorem{rem}[thm]{Remark}

\begin{document}

\author{Du\v san Dragutinovi\'c}
\address{ Mathematical Institute, Utrecht University,
P.O. Box 80010, 3508 TA Utrecht, The Netherlands}
\email{d.dragutinovic@uu.nl}

\title[Oort's conjecture and automorphisms of supersingular curves of genus four]{Oort's conjecture and automorphisms of supersingular curves of genus four}


\begin{abstract}
We show that every component of the locus of smooth supersingular curves of genus $4$ in characteristic $p>2$ has a trivial generic automorphism group. As a result, we prove Oort's conjecture about automorphism groups of supersingular abelian fourfolds for $p>2$. Our main idea consists of estimating dimensions of the loci of smooth supersingular curves that admit an automorphism of prime order by considering possible choices of the corresponding quotient curves. This reasoning also results in a new proof of Oort's conjecture for $g = 3$ and $p>2$, previously proved by Karemaker, Yuboko, and Yu.
\end{abstract}

\maketitle

\section{Introduction}

Let $k$ be a field of characteristic $p>0$ and $E$ be a supersingular elliptic curve over $k$, that is $E[p](\bar{k}) = \{O\}$. For any $g\geq 2$, we say that a $g$-dimensional abelian variety $A$ over~$k$ is \textit{supersingular} if there is a $\Bar{k}$-isogeny  
$A\otimes \bar{k} \overset{\sim}{\to} E^g\otimes \bar{k}$. Denote by $\cA_g = \cA_g\otimes \F_p$ the moduli space of principally polarized abelian varieties of dimension $g$, and let $\cS_g \subset \cA_g$ be the closed locus of all supersingular $g$-dimensional abelian varieties. It is well-known that $\cA_g$ is a smooth irreducible stack of dimension $\dim \cA_g = \frac{g(g + 1)}{2}$. By  \cite[Theorem 4.9]{lioort}, we know that $\cS_g$ is pure of dimension $\left \lfloor \frac{g^2}{4} \right \rfloor$, while its number of irreducible components is given in terms of certain class numbers. Throughout the paper, by an automorphism, we mean a geometric automorphism. In \cite[Problem~4]{emo01}, Oort posed the following conjecture and open problem.

\begin{conj}[Oort's conjecture] For any $g\geq 2$ and a prime number $p>0$, every component of the supersingular locus $\cS_g \subseteq \cA_g$ has generic automorphism group $\{\pm 1\}$.    
\label{conj:oort}
\end{conj}

\begin{prob} Suppose that $g\geq 3$ and a $p$ is a prime number. Suppose there exists a supersingular curve of genus $g$ in characteristic $p$. Is it true that there exists a supersingular curve $C$ of genus $g$ over $\overline{\mathbb{F}}_p$ such that $\mathrm{Aut}(C)$ is trivial?
\label{prob:oort}
\end{prob}

So far, it is known that Oort's conjecture fails for $(g, p) = (2, 2)$ by \cite{ibukiyama} and for $(g, p) = (3, 2)$ by \cite{oort_hess}, while it holds when $g = 2$ or $g = 3$ for any $p>2$ by \cite{ibukiyama} and \cite{kyy}. Before we state our main result, let us briefly introduce some notions. 
To a stable curve $C$ we can associate its Jacobian $\cJ_C = \mathrm{Pic}^0(C)$ which is a semi-abelian variety. This induces the Torelli morphism $j:\cMbar_g \to {\cAbar}_g$, where $\cMbar_g = \cMbar_g \otimes \F_p$ is the moduli space of stable curves of genus $g$ and ${\cAbar}_g$ is a fixed smooth toroidal compactification of $\cA_g$ as in \cite[Theorem 4.1]{alexeev}. We denote by $\cM_g$ the sublocus of $\cMbar_g$ consisting of smooth curves, and recall the well-known fact $\dim \cM_g = 3g - 3$. 
We say that a stable curve $C$ is supersingular if its Jacobian is supersingular; $C$ is then necessarily of compact type, i.e., its dual graph is a tree. Denote by $\cM_g^{ss}$ the locus of smooth supersingular curves of genus~$g$ in characteristic $p>0$. In Section \ref{sec:invariants}, we give an overview of all invariants in characteristic $p>0$ that we will use throughout the paper. In Section \ref{sec:proof}, we prove the following theorem, our main result.

\begin{thm} For $p>2$, every component of $\cM_4^{ss}$ has a trivial generic automorphism group.    
\label{thm:genus4}
\end{thm}
As an immediate consequence, we get an answer to Open problem \ref{prob:oort} for $g = 4$ and $p>2$. 

\begin{cor} For $p>2$, there exists a smooth supersingular curve of genus $4$ over $\overline{\mathbb{F}}_p$ whose automorphism group is trivial.     
\end{cor}
\begin{proof}
As it was shown in \cite[Corollary 1.2]{khs} by construction and later in \cite[Theorem~1.1]{pries23} using a geometric argument, there exists a smooth supersingular curve of genus~$4$ over $\overline{\mathbb{F}}_p$ for any $p>0$. In other words, $\cM_{4}^{ss}$ is always non-empty, so that Theorem \ref{thm:genus4} implies the result.    
\end{proof}

The proof of Theorem \ref{thm:genus4} consists of estimating dimensions of the loci of smooth supersingular curves $C$ that admit an automorphism of prime order $l>0$. For any such automorphism $\iota$, we get the quotient curve $D = C/\<\iota\>$, which is a smooth curve of some genus $g_D$. In Section \ref{sec:double0}, we present some known results about the case $g_D = 0$ for any $l>0$, and in Sections \ref{sec:double1} and \ref{sec:double2}, we pay special attention to the case $l = 2$ with $g_D = 1$ or $g_D = 2$. As another consequence of the preceding theorem, we have the following result.

\begin{cor}
Oort's conjecture \ref{conj:oort} holds for $g = 4$ and $p>2$.  
\label{cor:main_oort_a4}
\end{cor}

\begin{proof}
By \cite{igusa}, the Torelli locus $\cJ_4 = j(\cMbar_4)\cap \cA_4$ is an ample divisor in $\cA_4$. Therefore, similarly as in \cite[Proposition 2.4]{shankartsimerman}, we observe that $\cJ_4$ has a non-empty intersection with every positive dimensional closed subvariety of $\cA_4$. In particular, $\cJ_4$ intersects every component $\Gamma$ of the supersingular locus $\cS_4 \subseteq \cA_4$. 

A dimension count tells us that every generic point of $\cJ_4 \cap \Gamma$ represents a Jacobian of a smooth non-hyperelliptic curve $C$. Indeed, recall that the codimension of an intersection of subvarieties of $\cA_4$ is at most the sum of the codimensions. Therefore, as was already concluded in \cite[Theorem~1.1]{pries23}, the dimension of any component of $\cJ_4 \cap \Gamma$ is at least~$3$, while the locus $\Pi$ of products of supersingular Jacobians of lower dimension in $\cA_4$ has dimension $\dim \Pi = 2$. This implies that every generic point of the intersection corresponds to a Jacobian of a smooth curve $C$. The curve $C$ is non-hyperelliptic by Lemma~\ref{lem:oort_hess} below.

By Torelli's theorem \cite[Theorem 12.1]{milne}, for a smooth non-hyperelliptic curve $C$, it holds that $\mathrm{Aut}(\cJ_C) \cong \mathrm{Aut}(C)\times \{\pm 1\}$. Therefore, if there is a component of the supersingular locus $\cS_4$ whose generic automorphism group strictly contains $\{\pm 1\}$, it would follow that there is a component of the supersingular locus of smooth curves $\cM_4^{ss}$, whose generic automorphism group is non-trivial. In Theorem \ref{thm:genus4}, we show this cannot be the case when $p>2$, so the result follows. 
\end{proof}

We were informed that another proof of Oort's conjecture \ref{conj:oort} for $g = 4$ was previously obtained by Karemaker and Yu (\cite{vkcfy}), using different techniques than ours. Note that Corollary \ref{cor:main_oort_a4} does not imply Theorem \ref{thm:genus4}. Namely, Oort's conjecture~\ref{conj:oort} might hold for some pair $(g, p)$, while all smooth supersingular curves of genus $g$ in characteristic $p$ have non-trivial automorphism groups, giving a negative answer to Open problem \ref{prob:oort} for this pair $(g, p)$. See also Remark~\ref{rem:char2}, where we comment on the case $g = 4$ and $p = 2$.

In Section \ref{sec:genus3}, we apply our ideas described above to present a new proof of Oort's conjecture~\ref{conj:oort} in the case $g = 3$ and $p>2$. Finally, we present a proof we have not found in the literature of a well-known result by Oort in \cite{oort_hess} that there are no smooth supersingular hyperelliptic curves of genus $3$ in characteristic $2$. 

\subsection*{Acknowledgment}
The author is grateful to Carel Faber and Valentijn Karemaker for all the discussions and valuable comments and to Arizona Winter School, his project group, Steven Groen, and Rachel Pries for helpful conversations about double covers and motivation for studying them.  The author is supported by the Mathematical Institute of Utrecht University.

\section{Invariants in characteristic $p>0$}
\label{sec:invariants}
From now on, let $k$ be an algebraically closed field of characteristic $p>0$, and denote by $\sigma: k \to k$ the Frobenius of $k$.  The $\prank$ of a stable curve $C$ of genus $g\geq 2$, denoted by $f_C$, is the semisimple rank of the $\sigma$-linear Frobenius operator $F$ on $H^{1}(C, \O_C)$, i.e., $$f_C = \rank_{k}F^{g}|_{H^1(C, \O_C)};$$ note that $0 \leq f_C \leq g$. This definition generalizes the well-known notion of the $p$-rank of an abelian variety $A$ over $k$ as the number $f_A$ such that $\# A[p](k) = p^{f_A}$ for the choice $A = \cJ_C$, where $C$ is a smooth curve and $f_C = f_{\cJ_C}$. It holds that $$C \text{ is supersingular } \implies f_C = 0,$$
while the reverse implication does not hold for $g \geq 3$. By \cite[Theorem 2.3]{fabervdgeer}, the locus $V_f\cMbar_g$ of all stable curves $C$ of genus $g$ with $f_C \leq f$ is pure of dimension 
\begin{equation}
\dim V_f\cMbar_g = 2g - 3 + f.
\label{eqn:fvdg}
\end{equation}

Let $A$ be a principally polarized abelian variety (p.p.a.v.) over $k$ of dimension $g$. First, let $A[p^{\infty}] = \dirlim A[p^n]$ be its $p$-divisible group. By the Dieudonn\'e-Manin classification \cite{manin}, there are certain $p$-divisible groups $G_{c, d}$ for $c, d\geq 0$ relatively prime integers, such that there is an isogeny of $p$-divisible groups 
\begin{equation}
A[p^{\infty}] \sim \bigoplus_{\lambda = \frac{d}{c + d}}G_{c, d}
\label{eqn:dm_class}
\end{equation}
for a unique choice of \textit{slopes} $\lambda$. The \textit{Newton polygon} of $A$ is the collection of all slopes $\lambda$ that occur in \eqref{eqn:dm_class}, counted with multiplicities. $A$  is supersingular if and only if $A[p^{\infty}] \sim (G_{1, 1})^g$. 

Now, let $A[p]$ be the polarized $p$-torsion group subscheme of $A$ and let $\mathbb{D}(A) = \mathbb{D}(A[p])$ be its polarized Dieudonn\'e module.  By \cite[Section 5]{oort} or \cite[Section 2]{vdgeercycle}, there is a $1$-to-$1$ correspondence between the isomorphism classes of all $\mathbb{D}(A)$ and the Young diagrams $$\mu = [\mu_1, \mu_2, \ldots, \mu_n],$$ for some $n \in \{0, 1, \ldots, g\}$ and $\mu_i$ for $i \in \{1, \ldots, n\}$ with $g\geq \mu_1 > \ldots \mu_n > 0$. ($\mu = \o$ if $n = 0$.) We call the $\mu$ associated with $A$ via the correspondence above the \textit{Ekedahl-Oort type} of $A$ and write $\mu(A) = \mu$. Denote by $Z_{\mu}$ the locus in $\cA_g$ consisting of (isomorphism classes of) all p.p.a.v.'s $A$ with $\mu(A) = \mu$, and by $\overline{Z}_{\mu}$ its closure in $\cA_g$. By \cite[Corollary 11.2, Theorem 1.2]{oort}, each $Z_{\mu}$ is a locally closed subset, pure of codimension $\sum_{i = 1}^n \mu_i$ in $\cA_g$, and it is quasi-affine, so it cannot contain a positive-dimensional complete subvariety. 
Furthermore, by \cite[Proposition 11.1]{oort}, if we introduce a partial ordering by $$\mu = [\mu_1, \ldots, \mu_n] \geq \mu'= [\mu'_1, \ldots, \mu'_m] \text{ if } n\leq m \text{ and } \mu_i \leq \mu_i' \text{ for all } 1\leq i\leq n, $$ as long as $g  \geq \max\{\mu_1, \mu_1'\}$, we get 
\begin{equation*}
\mu' \leq \mu \quad \implies Z_{\mu'} \subseteq \overline{Z}_{\mu} \text{ in } \cA_g.
\end{equation*}
If $\mu(A) = [\mu_1, \mu_2, \ldots, \mu_n]$, then we can compute the $p$-rank of $A$ as $f_A = g - \mu_1$. Finally, if $\mu(A) = [g, g-1, \ldots, 2, 1]$, then $A\cong E^g$ with $E$ a supersingular elliptic curve and we say that such an $A$ is \textit{superspecial}. A stable curve $C$ is \textit{superspecial} if its Jacobian $\cJ_C$ is superspecial; note that there are at most finitely many smooth superspecial curves of fixed genus $g\geq 2$ in characteristic $p>0$. 
Here are some properties when $g = 4$. 

\begin{lem}[{\cite[Proposition 5.13]{iky},  \cite[4.4]{pries_eo_classification}, \cite[Remark 5.13]{elkinpries}}]
The Ekedahl-Oort strata in $\cA_4$ whose elements have $p$-rank $0$ are precisely $Z_{\mu}$ for those $\mu$ appearing in the diagram below,  
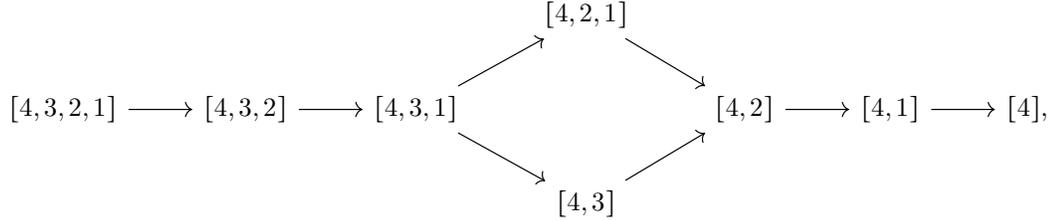
\begin{small}
\begin{center}
\begin{tikzcd}
	&&& {[4, 2, 1]} \\
	{[4, 3, 2, 1]} & {[4, 3, 2]} & {[4, 3, 1]} && {[4, 2]} & {[4, 1]} & {[4]}, \\
	&&& {[4, 3]}
	\arrow[from=2-1, to=2-2]
	\arrow[from=2-2, to=2-3]
	\arrow[from=2-3, to=1-4]
	\arrow[from=2-3, to=3-4]
	\arrow[from=3-4, to=2-5]
	\arrow[from=1-4, to=2-5]
	\arrow[from=2-5, to=2-6]
	\arrow[from=2-6, to=2-7]
\end{tikzcd}  
\end{center}
\end{small}
where there is an arrow $\mu \to \mu'$ if and only if $Z_{\mu} \subseteq \overline{Z_{\mu'}}$. Furthermore, let $A$ be any p.p.a.v. of dimension $4$ whose Ekedahl-Oort type equals $\mu$: 
\begin{itemize}
    \item if $\mu \in \{[4], [4, 1], [4, 3, 1]\}$, then $\mathbb{D}(A)$ is indecomposable, i.e., it is not a product $\mathbb{D}(A_1) \oplus \mathbb{D}(A_2)$ for any two positive-dimensional p.p.a.v.'s $A_1$ and $A_2$;
    \item if $\mu = [4, 2]$, then $\mathbb{D}(A) \cong \mathbb{D}(A_1)\oplus \mathbb{D}(A_2)$ with $A_1$ and $A_2$ two p.p.a.v.'s such that $\dim A_1 = 1$, $\dim A_2 = 3$, $\mu(A_1) = [1]$, and $\mu(A_2) = [3]$;
    \item if $\mu = [4, 3]$, then $\mathbb{D}(A) \cong \mathbb{D}(A_1)\oplus \mathbb{D}(A_1)$ with $A_1$ a p.p.a.v. with $\dim A_1 = 2$ and $\mu(A_1) = [2]$;
    \item if $\mu = [4, 2, 1]$, then $\mathbb{D}(A) \cong \mathbb{D}(A_1)\oplus \mathbb{D}(A_2)$ with $A_1$ and $A_2$ two p.p.a.v.'s such that $\dim A_1 = 1$, $\dim A_2 = 3$, $\mu(A_1) = [1]$, and $\mu(A_2) = [3, 1]$.
\end{itemize}
\label{lemma:eo_genus4}
\end{lem}

\section{Loci of curves with automorphisms and quotient curves}
\label{sec:pre}

Let $C$ be a smooth curve of genus $g$ over an algebraically closed field $k$ of characteristic $p>0$ and assume that there exists an automorphism $\gamma$ of $C$, which is of prime order $l > 0$. For a moment, let us assume that $l\neq p$. Denote $G = \<\gamma\> \cong \Z/l\Z$. Then, there exists $D = C/G$, the quotient of $C$ by $G$, which is a smooth curve. Its function field $\kappa(D)$ is the subfield of $\kappa(C)$ fixed by $G$, i.e., $\kappa(D) = \kappa(C)^{G}$. Consider the canonical map $$C \to D = C/G.$$ Its local description tells us that the ramification points of this map are exactly the points of $C$ with a non-trivial stabilizer in $G$ so that the ramification indices are either $1$~or~$l$. In other words, this map is a $(\Z/l\Z)$-cover. The Riemann-Hurwitz formula tells us that $2g_C - 2 = l(2g_D - 2) + |B|(l - 1)$, where $g_C$ and $g_D$ are the genera of $C$ and $D$, and $B$ is the set of branch points. Given a smooth curve $D$ of genus $g_D\geq 0$ and a finite set $B$ of points on $D$, there are (at most) finitely many $(\Z/l\Z)$-covers $C \to D$ whose sets of branch points coincide with $B$. One computes the genus of $C$ using the mentioned formula.

We recall the following bound on the dimension of the locus of curves with automorphisms obtained by Achter, Glass, and Pries using the description above, which also holds for $l = p$. 

\begin{lem}[{\cite[Lemma 2.1]{agp}}] Let $l>0$ be any prime number, and let $\cM_g^l \subseteq \cM_g$ be the closed locus of curves that admit an automorphism of order $l$ and suppose that $\Gamma_l$ is an irreducible component of $\cM_g^l$ with generic point $\eta$. If $D$ is the quotient of $C_{\eta}$ by a group of order $l$, and $g_D$ and $f_D$ are respectively its genus and $\prank$, then $$\dim \Gamma_l \leq 2(g - g_D)/(l - 1) + f_D - 1.$$     
\label{lem:agp_dim}
\end{lem}

\subsection{Cyclic covers of the projective line}
\label{sec:double0}

In this subsection, we consider the cyclic covers of the projective line $\P^1$ of degrees $2$, $3$, and $l = p$, which correspond to hyperelliptic, cyclic trigonal, and Artin-Schreier curves respectively. Recall that there is a unique projective automorphism that sends any triple of distinct points on $\P^1$ to the triple $(0, 1, \infty)$.

First, we mention a result about smooth \textit{hyperelliptic curves} $C$, i.e., the curves with an automorphism $\iota \in  \mathrm{Aut}(C)$ of degree $2$, such that $C/\<\iota\> \cong \P^1$. Let $\cH_g$ denote the locus of all smooth hyperelliptic curves of genus $g$ in characteristic $p>2$. It is well known that $$\dim \cH_g = 2g - 1,$$ which corresponds to the choice of $2g + 2$ branch points of the double cover $C \to \P^1$ with $C$ a smooth hyperelliptic curve of genus $g$, and fixing three of the branch points. Similarly to \eqref{eqn:fvdg}, it was shown in \cite[Theorem 1]{glasspries} that the dimension of any $p$-rank $\leq f$ locus $V_f\cH_g$ of $\cH_g$ equals $$\dim V_f\cH_g = g - 1 + f.$$ Denote by $\cH_g^{ss}$ the locus of supersingular curves in $\cH_g$ (possibly empty for $g\geq 4$) and recall $\cH_{g}^{ss} \subseteq V_0\cH_g$. The following result is due to Oort for $g = 3$ and Achter-Pries for $g\geq 4$.

\begin{lem}[{\cite[Theorem 1.12]{oort_hess}}, {\cite[Corollary 3.16]{ap}}, $p>2$] Let $\cH_g^{ss}$ be the locus of smooth hyperelliptic curves of genus $g\geq 3$ which are supersingular. Then $\dim \cH_3^{ss} = 1$ and $\dim \cH_g^{ss} \leq g - 2$ for $g\geq 4$.
\label{lem:oort_hess}
\end{lem}

Denote by $\mathcal{T}_g$ the locus of smooth \textit{cyclic trigonal} curves of genus $g$, i.e., of smooth curves~$C$ of genus $g$ with an automorphism $\iota \in  \mathrm{Aut}(C)$ of degree $3$ such that $C/\<\iota\> \cong \P^1$. Denote $G = \<\iota\> \cong \Z/3\Z$. Assume that $p>0$ and $p \neq 3$. We can see that the cover $C \to C/G \cong \P^1$ is ramified at $g + 2$ points so that $\mathcal{T}_g$ is pure of dimension $$\dim \mathcal{T}_g =  g - 1.$$ 
Let $\xi: \{1, 2, \ldots, g + 2\} \to (\Z/3\Z)^*$ be a map of sets such that $\sum_{i = 1}^{g + 2} \xi(i) = 0 \in \Z/3\Z$ called the class vector and let $\overline{\xi} = (|\xi^{-1}(1)|, |\xi^{-1}(2)|)$ be its inertia type. By \cite[Lemma~2.3]{ap_tri}, for any irreducible component of the locus~$\mathcal{T}_{g, \mathrm{{ord}}}$ of smooth cyclic trigonal curves of genus~$g$ with ordered ramification points, there is a class vector $\xi$ as above. We denote this component by~$\mathcal{T}_{g, \mathrm{ord}}^{\xi}$. By forgetting the labeling, we get the map $\mathcal{T}_{g, \mathrm{ord}}^{\xi} \to \mathcal{T}_{g}^{\overline{\xi}}$ with $\mathcal{T}_{g}^{\overline{\xi}}\subseteq \mathcal{T}_g$. We see that~$\mathcal{T}_g$ consists of the union of such $\mathcal{T}_{g}^{\overline{\xi}}$ and note that the choice of $\xi$ corresponds to encoding the action of $G$ on the tangent spaces of $C$ at the ramification points. For more details, see \cite[Section 2.1]{ap_tri}.

\begin{lem}[\cite{ap_tri}, $p>0 \text{ and } p\neq 3$]
Every component $\Gamma$ of the locus $\mathcal{T}_g^{ss}$ of smooth supersingular cyclic trigonal curves of genus $g\geq 3$ has dimension $\dim \Gamma \leq g - 2$.
\label{lem:triel}
\end{lem}
\begin{proof}
By the proof of \cite[Corollary 3.11]{ap_tri}, every component $\mathcal{T}_{g, \mathrm{ord}}^{\xi}$ of $\mathcal{T}_{g, \mathrm{ord}}$ contains an element whose underlying smooth curve $C$ satisfies that its Jacobian $\cJ_C$ is absolutely simple. This curve $C$ cannot be supersingular so that no generic point of $\mathcal{T}_g$ is supersingular. Therefore, $\dim \mathcal{T}_g^{ss} < \dim \mathcal{T}_g$, which proves the result.  
\end{proof}

Finally, we consider cyclic covers of the projective line $\P^1$ of degree $l = p$. A smooth curve~$C$ over $k$ that is a $(\Z/p\Z)$-cover of a projective line is usually referred to as an \textit{Artin-Schreier} $k$-curve. Denote by $\mathcal{AS}_g$ the locus of all smooth Artin-Schreier $k$-curves of genus $g\geq 2$, by $\mathcal{AS}_{g, f} \subseteq \mathcal{AS}_g$ the locus of Artin-Schreier $k$-curves $C$ with $p$-rank exactly $f_C = f$, and by $\mathcal{AS}_g^{ss} \subseteq \mathcal{AS}_g$ the supersingular locus of $\mathcal{AS}_g$. The $p$-rank stratification of $\mathcal{AS}_g$ was studied in \cite{prieszhu}. It turns out that $g = d(p-1)/2$ and $f = r(p - 1)$ for some $d\geq 1$ and $r\geq 0$; otherwise, $\mathcal{AS}_{g, f}$ is empty. For any eligible $0\leq f \leq g$, Pries and Zhu compute the number of irreducible components of $\mathcal{AS}_{g, f}$ and determine their dimensions. Below, we extract conclusions about the case $f = 0$ from their results. See also \cite[Lemma 2.2]{agp}.

\begin{lem}[{\cite[Theorem 1.1]{prieszhu}}] Let $g = d(p-1)/2$ for some $d\geq 1$ and let $f = 0$. Then, there is a unique component of the $p$-rank $0$ locus $\mathcal{AS}_{g, 0}$ if $d + 1 \equiv 0 \text{ }\mathrm{ mod }\text{ } p$, while $\mathcal{AS}_{g, 0}$ is empty otherwise. In the former case, it holds that $$\dim \mathcal{AS}_{g, 0} = d - 1 - \left \lfloor \frac{d + 1}{p} \right \rfloor.$$ 
In particular, for $p = 3$ we find that $\dim \mathcal{AS}_{3, 0} = 1$ and $\dim \mathcal{AS}_{4, 0} = 2$, and therefore $\dim \mathcal{AS}_{3}^{ss} \leq 1$ and $\dim \mathcal{AS}_{4}^{ss} \leq 2$.
\label{prieszhu}
\end{lem}

\subsection{Double covers of curves of genus $1$}
\label{sec:double1}
Let $E$ be a fixed elliptic curve in characteristic $p>2$, that is a $1$-pointed curve of genus $1$. 
Denote by $\cM_{E, g}$ the locus of smooth curves $C$ of genus $g$  which are double covers of $E$, i.e., such that there exists a map $\pi: C \to E$ of degree~$2$. 
For $0\leq f\leq g$, denote by $V_f\cM_{E, g}$ the locus of curves $C$ in $\cM_{E, g}$ with $f_C \leq f$. 
These loci were studied in \cite{aws}. In particular, for $E$ a supersingular elliptic curve and for any $g\geq 3$ and $f$ with $0 \leq f \leq g - 1$, it was shown that the locus $V_f{\cM}_{E, g}$ is pure of dimension 
\begin{equation}
\dim V_f{\cM}_{E, g} = g - 2 + f.
\label{eqn:aws}
\end{equation}

Denote by $\cB_g$ the (bielliptic) locus of smooth curves $C$ over $k$ of genus $g$ for which there exists an elliptic curve $E$ and a map $\pi: C \to E$ of degree $2$. By the Riemann-Hurwitz formula, this map has $2g - 2$ branch points $b_1, \ldots, b_{2g - 2}$. Given an elliptic curve $E = (E, b_1)$ (up to translation, we can always assume that $b_1$ is the neutral element of~$E$), and $2g - 1$ points $b_2, \ldots, b_{2g - 2}$, there are only finitely many double covers $\pi: C \to E$ branched over $b_1, \ldots, b_{2g - 2}$. They correspond to the choice of a line bundle $L$ on $E$ such that $L^2 \cong \O_D(- \sum_{i = 1}^{2g - 2}b_i)$ as was discussed in \cite[Section 1]{mumford}. This implies $$\dim \cB_g = 2g - 2.$$ 
To present an alternative proof of Lemma \ref{lem:aws} below, we will need the notion of admissible double covers, that we adapt from \cite[Section 4.1]{acv}, similarly as in \cite[Section 2]{cp}.

\begin{dfn}
We say that $C$ is an \textit{admissible bielliptic} curve of genus $g$ if $C$ is a nodal curve of genus $g$ together with a map $\pi: C \to D$ of degree $2$, where $(D, b_1, \ldots, b_n)$ is an $n$-pointed stable curve of genus $1$ such that 
\begin{itemize}
    \item $\pi$ is finite and maps every node of $C$ to a node of $D$;
    \item the restriction $\pi^{sm}: C^{sm} \to D^{sm}$ of the map $\pi$ to the smooth loci of $C$ and $D$ is branched exactly at the marked points $b_1, \ldots, b_n$.
\end{itemize}    
\end{dfn}
Formally, we should write $C = (C, \pi: C \to (D, b_1, \ldots, b_n))$. Similarly, one can define the notions of families of admissible bielliptic curves and their isomorphisms; see \cite[Section 4.1]{acv}. This results in a proper Deligne-Mumford stack $\overline{\cB}_{g}^{adm}$ of admissible bielliptic curves of genus $g$ in characteristic $p\neq 2$. Consider the map $(C, \pi: C \to (D, b_1, \ldots, b_n)) \mapsto \overline{C}$, which is obtained by forgetting the admissible structure on $C$ and then contracting its rational components that meet the remaining components of $C$ in at most $2$ points to get~$\overline{C}$. This map induces a morphism $\overline{\cB}_{g}^{adm} \to \cMbar_g,$ whose image is the closed locus $\overline{\cB}_g$, the closure of $\cB_g$ in $\cMbar_g$. In \cite[Section 3.4]{svz}, one can find the description of the boundary $\Delta_g = \overline{\cB}_g - {\cB}_g$ in terms of stable graphs, using that every degree $2$ cover is automatically a $\Z/2\Z$-cover, so that $\overline{\cB}_{g}^{adm} \cong \overline{\cH}_{g, \Z/2\Z, (1^{2g - 2})}$, with $ \overline{\cH}_{g, G, \xi}$ the stack of admissible $G$-covers of genus $g$ with monodromy data $\xi$ considered in \cite{svz} (here with $G = \Z/2\Z$ and $\xi = (1^{2g - 2})$).

Denote by $(C_1, q_1)\cup_{q_1 = q_2} (C_2, q_2)$ the clutch of curves $C_1$ and $C_2$ with points $q_1 \in C_1$ and $q_2 \in C_2$ identified. Finally, denote by $\cB_g^{ss}$ the locus of all supersingular curves in $\cB_g$. We have the following result.

\begin{lem}[$p>2$] Every component $\Gamma$ of the locus $\cB_3^{ss}$ has dimension $\dim \Gamma \leq 1$, while every component $\Gamma'$ of the locus $\cB_4^{ss}$ has dimension $\dim \Gamma' \leq 2$.
\label{lem:aws}
\end{lem}
\begin{proof}
Since there are only finitely many supersingular elliptic curves up to $k$-isomorphism, the formula \eqref{eqn:aws} obtained in \cite{aws} implies the result. For the reader's convenience, we offer an alternative proof here. Let $\overline{\cB}_g$ be the closure of $\cB_g$ as above, and denote $\Delta_g = \overline{\cB}_g - {\cB}_g$.

First, we prove that there are only finitely many isomorphism classes of smooth bielliptic curves $C$ of genus $2
$ with $f_C = 0$, i.e., that $\cB_2^{ss}$ is finite. Indeed, since the hyperelliptic involution $\iota$ of $C$ is unique, if we denote by $\theta \neq \iota$ a bielliptic involution on $C$, then $E = C/\<\theta\>$ and $E' = C/\<\iota \circ \theta\>$ are two supersingular elliptic curves and $C \to E \times E'$ is an isogeny of degree $2$, coprime to $p>2$. Therefore, $C$ would be a smooth superspecial curve, which shows the desired conclusion. We note that the locus $\cB_2^{ss}$ is empty for $p = 3$. This follows from \cite[Theorem~1.1]{ekedahl}, saying that there are no smooth superspecial curves of genus $2$ in characteristic $3$. (Formally, to cover this case $p= 3$ below and degenerations of the cases we present, we also use that the supersingular locus $\Delta_2^{ss}$ of $\Delta_2 = \overline{\cB_2} - \cB_2$ is finite.) 

Let $\Gamma$ be a component of $\cB_3^{ss}$ and let $\overline{\Gamma}$ be its closure in $\overline{\cB}_3$. 
Note that $\Gamma$ is quasi-affine. Namely, if $C$ is a smooth curve of genus $3$, $E$ a genus $1$ curve, $C \to E$ a double cover whose branch locus equals $B = \{b_1, b_2, b_3, b_4\}$, then $(E, b_1)$ is an elliptic curve and $b_2, b_3, b_4$ are three distinct points in $E - \{b_1\}$, and we could think of $E - \{b_1\}$ as of an affine curve. Using this property of $\Gamma$, we find that $\overline{\Gamma} \cap \Delta_3\neq \o$.
Now, the smoothness of $\cMbar_3$ implies that $\dim \Gamma \leq \dim (\overline{\Gamma} \cap \Delta_3) + 1.$ Therefore, it is enough to show that $\dim (\overline{\Gamma} \cap \Delta_3) = 0$. The explicit description of the boundary of the bielliptic locus using stable graphs from \cite[Section~3.4]{svz}, tells us that the irreducible components of $\overline{\Gamma} \cap \Delta_3$ must be either of the form~$\Delta'$ or $\Delta''$ as below:
\begin{itemize}
    \item $\Delta'$ is the closure of the locus of stable curves of the form $(C_1', q_1)\cup_{q_1 = q_2} (C_2', q_2)$ where $C_1' \in \cB_2^{ss}$ and $q_1 \in C_1'$ is one of the ramification points of the morphism $C_1' \to E$ for some elliptic curve $E$, while $(C_2', q_2)$ is a supersingular elliptic curve. (Formally, it could also happen $C_1' \in \Delta_2^{ss}$ with $q_1\in (C_1')^{sm}$ a Weierstrass point of one of the components of $C_1'$, which is the case that we cover similarly.)
    \item $\Delta''$ is the closure of the locus of stable curves of the form $$(C_1'', q)\cup_{q = q_1} (C_2'', q_1, q_2) \cup_{q_2 = q} (C_1'', q)$$ where $(C_1'', q)$ and $(C_2'', q_1)$ are supersingular elliptic curves, and $q_2 \neq q_1$ is the image of $q_1$ under the involution $\iota$ of $C_2''$ (we can always assume $\iota (q_1) \neq q_1$). 
\end{itemize}
It is clear that the loci $\Delta''$ are finite. The loci $\Delta'$ are finite as well. This follows from our conclusion that $\cB_2^{ss}$ is finite and from the fact that for each $C_1' \in \cB_2$ there are only finitely points $q_1$ such that $q_1$ is a ramification point of a double cover $C_1' \to E$ for some elliptic curve~$E$. Therefore, $\dim \cB_3^{ss} \leq 1$.

A similar argument shows that $\dim \cB_4^{ss} \leq 2$. We mention a few details. Any component~$\Gamma'$ of $\cB_4^{ss}$ is quasi-affine, so that $\overline{\Gamma'} \cap \Delta_4 \neq \o$ and $\dim \Gamma' \leq \dim (\overline{\Gamma'} \cap \Delta_4) + 1$. Now, there are three possible forms of curves in $\overline{\Gamma'} \cap \Delta_4$. In each of the three cases, we find that these loci are at most $1$-dimensional, using that $\cB_2^{ss}$ is finite and $\dim \cB_3^{ss} \leq 1$. This shows $\dim (\overline{\Gamma'} \cap \Delta_4) \leq 1$ and gives the result. 
\end{proof}

\subsection{Double covers of curves of genus $2$}
\label{sec:double2}

In the rest of this section, we consider the (bi-$2$) locus $\cD_4$ in $\cM_4$ in characteristic $p>2$, consisting of all smooth curves $C$ of genus $4$ such that there is a double cover 
\begin{equation}
\pi: C \to D,    
\label{eqn:CtoD}
\end{equation}
for some smooth curve $D$ of genus $2$. The Riemann-Hurwitz formula tells us that the double cover \eqref{eqn:CtoD} has $2$ branch points, $b_1$ and $b_2$. Given a curve $D$ and two distinct points $b_1, b_2 \in D$, there are only finitely many double covers \eqref{eqn:CtoD} and they correspond to the choice of a line bundle $L$ on $D$ such that $L^2 \cong \O_D(-b_1 - b_2)$ as was discussed in \cite[Section~1]{mumford}. Therefore, using that $\dim \cM_2 = 3$, we can conclude that $\dim \cD_4 = 5$. Let us denote by $\cD_4^{ss}$ the locus of all supersingular curves in $\cD_4$. Since the locus $\cM_2^{ss}$ of supersingular curves of genus $2$ is pure of dimension $1$ as $\overline{j(\cM_2^{ss})} = \cS_2 \subseteq \cA_2$, we find that $\dim \cD_4^{ss} \leq 3$. In fact, this section aims to prove that $\dim \cD_4^{ss} \leq 2$. We will need this estimate in the proof of Theorem~\ref{thm:genus4}.

By \cite[Corollary 1 and Corollary 2]{mumford}, there is a principally polarized abelian variety~$P$ of dimension $2$, the \textit{Prym variety} of $\pi$, such that there is an isogeny of degree~$2$, 
\begin{equation}
\cJ_C \sim \cJ_D \times P,    
\label{eqn:isogCD}
\end{equation}
with $\cJ_C$ and $\cJ_D$ the Jacobian varieties of $C$ and $D$. Therefore, $f_C = 0$ if and only if $C$ is supersingular, since $\dim \cJ_D = \dim P = 2$ so $f_P = 0 \Leftrightarrow
 P$ is supersingular and $f_D = 0 \Leftrightarrow
 D$ is supersingular.

\begin{lem}[$p>2$] Every component $\Gamma$ of the locus $\cD_4^{ss}$ of smooth supersingular curves of genus $4$ that are double covers of a genus $2$ curve has dimension $\dim \Gamma \leq 2$.
\label{lem:MD4}
\end{lem}

\begin{proof}
Let $\Gamma$ be any component of $\cD_4^{ss}$ and recall that $\dim \Gamma \leq 3$. Assume that $\dim \Gamma = 3$. Let $S = \overline{j(\Gamma)} \subseteq \cA_4$ be the closure of the image of $\Gamma$ under the Torelli map. Then $S$ is a closed $3$-dimensional family of supersingular abelian fourfolds. Moreover, since the degree of isogeny \eqref{eqn:isogCD} is $2$, which is coprime to $p > 2$, there is an isomorphism of polarized Dieudonn\'e modules 
\begin{equation}
\mathbb{D}(\cJ_C)\cong \mathbb{D}(\cJ_D) \oplus \mathbb{D}(P).    
\label{eqn:dieuprym}
\end{equation}
By Lemma \ref{lemma:eo_genus4}, this implies $j(\Gamma) \subseteq \overline{Z}_{[4, 3]}$, i.e., that $S = \overline{j(\Gamma)}$ is a component of $\overline{Z}_{[4, 3]}$ by comparing their dimensions. We claim that $S\cap Z_{[4, 3, 1]} = \o$. Indeed, since the Ekedahl-Oort type $[4, 3, 1]$ occurs only for indecomposable abelian varieties, it would follow that there is a smooth curve $C$ in $\Gamma$ with $\mu(\cJ_C) = [4, 3, 1]$ (see Lemma \ref{lemma:eo_genus4}). However, \eqref{eqn:dieuprym} tells us this cannot happen. Therefore $$S \subseteq Z_{[4, 3]}\cup Z_{[4, 3, 2]}\cup Z_{[4, 3, 2, 1]}.$$ Finally, using that $Z_{[4, 3]}$ and $Z_{[4, 3, 2]}$ are quasi-affine by  \cite[Theorem 1.2]{oort}, and $\dim S=3$ by our assumption, we would get that $\dim (S \cap \overline{Z}_{[4, 3, 2]})\geq 2$ and then $\dim (S \cap Z_{[4, 3, 2, 1]})\geq 1$, which is impossible since $S \cap Z_{[4, 3, 2, 1]} \subseteq Z_{[4, 3, 2, 1]}$ and $\dim Z_{[4, 3, 2, 1]} = 0$. Hence, $\dim \Gamma \leq 2$ and the result follows.  
\end{proof}

\section{Proof of Theorem \ref{thm:genus4}}
\label{sec:proof}

Let $p>2$ be any prime number, let $\Gamma$ be any component of $\cM_4^{ss}$, and let $C$ be a smooth curve corresponding to the generic point of $\Gamma$. Suppose that $\mathrm{Aut}(C)$ contains an automorphism $\gamma$ which is of prime order $l>1$ and denote $G = \<\gamma\> \subseteq \mathrm{Aut}(C)$. Let $D$ be the quotient curve $C/G$, let $g_D\geq 0$ be its genus, and note that the $p$-rank of $D$ equals $f_D = 0$ since $C$ is supersingular. Below, we consider the canonical map $$C \to D = C/G$$ for all possible choices of $l$ and $g_D$ and discuss the outcome. 
\begin{enumerate} 

    \item \label{itm:1} If $l>3$ or $l = 3$ and $g_D \geq 1$, then Lemma \ref{lem:agp_dim} tells us that $\dim \Gamma \leq 2.$ 
    \item \label{itm:2} If $l = 3$ and $g_D = 0$ so that $D \cong \P^1$ then $C$ is either a cyclic trigonal curve if $p \neq 3$ or it is an Artin-Schreier $k$-curve if $l = p = 3$; we discuss this in Section \ref{sec:double0}. In the first case $p \neq 3$, it holds that $\Gamma \subseteq \mathcal{T}_4^{ss}$, so that $\dim \Gamma \leq 2$ by Lemma~\ref{lem:triel}. Otherwise, $l = p = 3$ and $\Gamma \subseteq \mathcal{AS}_{4}^{ss}$ so that $\dim \Gamma \leq 2$ by Lemma \ref{prieszhu}.
    
    \item \label{itm:3} Finally, let us assume that $l = 2$ and show that $\dim \Gamma \leq 2$ in all possible cases.
    \begin{itemize}
    \item  If $g_D = 0$, then $C \to D \cong \P^1$ would be a hyperelliptic curve. This would imply that $\Gamma$ is a component of $\cH_4^{ss}$. Lemma \ref{lem:oort_hess} tells us that $\dim \Gamma \leq 2$. 
    \item If $g_D = 1$, then $C \to D$ is a double cover of $D$ whose $\prank$ equals $0$. Therefore, $\Gamma \subseteq \cB_4^{ss}$, with $\cB_4^{ss}$ the locus of all smooth supersingular bielliptic curves of genus~$4$; we considered these loci in Section \ref{sec:double1}. Lemma \ref{lem:aws} implies that $\dim \cB_4^{ss} \leq 2$, so it would again follow that $\dim \Gamma \leq 2$. 
    
    \item The final possibility is $g_D = 2$, when $C \to D$ is a double cover of a smooth curve~$D$ of genus $2$. As a consequence of our discussion in Section~\ref{sec:double2} we have that $\Gamma \subseteq \cD_4^{ss}$. Therefore, Lemma~\ref{lem:MD4} tells us that $\dim \Gamma \leq 2$. 
    \end{itemize}
\end{enumerate}
However, none of the previously considered cases can happen because any irreducible component $\Gamma$ of $\cM_4^{ss}$ has $\dim \Gamma \geq 3$, as remarked in the proof of Corollary \ref{conj:oort}. Namely, $\dim \Gamma = \dim \overline{j(\Gamma)}$, while $$\overline{j(\Gamma)} \subseteq \cJ_4 \cap \cS_4$$ is a component of $\cJ_4 \cap \cS_4$, so the smoothness of $\cA_4$ implies that $\dim \Gamma \geq 3$. We conclude $l \nmid |\mathrm{Aut}(C)|$ for every prime number $l>1$ and hence $\mathrm{Aut}(C)$ is a trivial group. This proves the theorem.

\begin{rem}
With the same argument, we can conclude that the generic automorphism groups of components of $\cM_4^{ss}$ in characteristic $p = 2$ cannot contain elements of prime order $l>2$. However, we cannot exclude the case $l = 2$ because of the assumptions on the loci of hyperelliptic curves and the loci considered in Sections \ref{sec:double1} and \ref{sec:double2}. 

In fact, Carel Faber informed us there is a component $\Gamma$ of the supersingular locus~$\cM_4^{ss}$ in characteristic $p = 2$, which generically consists of smooth curves $C$ of genus~$4$ that are double covers of genus $2$ curves. The automorphism group of any such curve $C$ contains an element of order $2$, so the generic automorphism group of~$\Gamma$ is not trivial. Thus, the assumption $p>2$ of Theorem \ref{thm:genus4} is necessary. However, it still might be the case that the generic automorphism group of the locus $\cS_4\subseteq \cA_4$ in characteristic $p = 2$ is $\{\pm 1\}$. 
\label{rem:char2}
\end{rem}

\section{Oort's conjecture in genus $3$}
\label{sec:genus3}

\begin{comment}
We have the following results.

\begin{prop} For $p>3$, generic automorphism group of any component of $\cM_2^{ss}$ equals $\{\pm 1\}$.  
\label{prop:genus2}
\end{prop}
\begin{proof}
Let $\Gamma$ be any component of $\cM_2^{ss}$ and let $C$ be the smooth hyperelliptic curve corresponding to the generic point of $\Gamma$. Recall $\dim \Gamma = 1$ and let $\iota$ be the hyperelliptic involution of $C$. Similarly as in the proof of Proposition \ref{prop:genus3}, Lemma \ref{lem:agp_dim} and Lemma \ref{lem:triel} tell us that $C$ does not admit an automorphism of prime order $l>2$. Finally, since the hyperelliptic involution $\iota$ is unique, if there is an automorphism $\theta \neq \iota$ of order $2$ in $\mathrm{Aut}(C)$, then $E = C/\<\theta\>$ and $E' = C/\<\iota \circ \theta\>$ would be two supersingular elliptic curves. Therefore, there would exist an isogeny $\cJ_C \sim E \times E'$ of degree $2 \neq p$ and $C$ would be a superspecial curve. Since the locus of superspecial curves is finite, we can finally conclude that $\mathrm{Aut}(C) = \<\iota\> \cong \{\pm 1\}$.  
\end{proof}
\end{comment}

In this section, we present new proofs of some known results about the loci of supersingular curves of genus $g = 3$. We start by proving an analog result of Theorem \ref{thm:genus4} for $g = 3$ with similar techniques.
\begin{prop}
For $p>2$, every component of $\cM_3^{ss}$ has a trivial generic automorphism group.    
\label{prop:genus3}
\end{prop}

\begin{proof}
We follow the argument of the proof of Theorem \ref{thm:genus4}. Let $\Gamma$ be any component of $\cM_3^{ss}$ and let $C$ be the smooth curve corresponding to the generic point of $\Gamma$. Suppose that $\mathrm{Aut}(C)$ contains an automorphism $\gamma$ which is of prime order $l>1$ and denote $G = \<\gamma\> \subseteq \mathrm{Aut}(C)$. Let $D$ be the quotient curve $C/G$ and note that $f_D = 0$ since $C$ is supersingular. We use that $\cS_3$ is pure of dimension $2$ so that $\dim \Gamma = 2$.

Below, we consider all possible choices of $l$ and $g_D$.
\begin{itemize}
    \item If $l>3$ or $l = 3$ and $g_D \geq 1$, then Lemma \ref{lem:agp_dim} tells us that $\dim \Gamma \leq 1$, which is impossible since $\dim \Gamma = 2$. 
    \item If $l = 3$ and $g_D = 0$, then $\Gamma \subseteq \mathcal{T}_{3}^{ss}$ if $p \neq 3$ or $\Gamma \subseteq \mathcal{AS}_{3}^{ss}$ if $l = p = 3$. Now, Lemma~\ref{lem:triel} and Lemma \ref{prieszhu} tell us that $\dim \Gamma \leq 1$ in both cases.
    \item Finally, assume that $l = 2$. If $g_D = 0$, then $\Gamma$ would be a component of $\cH_3^{ss}$, so that $\dim \Gamma = 1$ by Lemma \ref{lem:oort_hess}. If $g_D = 1$, then $\Gamma\subseteq \cB_3^{ss}$ and thus $\dim \Gamma \leq 1$ by Lemma \ref{lem:aws}. The final possibility is that $g_D = 2$ when $C \to D$ would be an \'etale double cover of a supersingular curve $D$ of genus $2$. Note that for any smooth curve $D$, there are only finitely many \'etale double covers $C \to D$ and they correspond to the $2$-torsion points of $\cJ_D$, i.e., to $\cJ_D[2]$. In this case, $\cJ_D[2](k) \cong (\Z/2\Z)^{4}$ as $p\neq 2$. Therefore, it follows that $\dim \Gamma \leq 1$ using that $\dim \cM_2^{ss} = 1$ as $\overline{j (\cM_2^{ss})} = \cS_2$ and $\dim \cS_2 = 1$.
\end{itemize}
Therefore, $l \nmid |\mathrm{Aut}(C)|$ for any prime number $l>1$, so that $\mathrm{Aut}(C)$ is a trivial group. 
\end{proof}

Recall that Oort's conjecture \ref{conj:oort} for $g = 3$ was proven in \cite{kyy} for any $p>2$. Their proof relies on geometry and arithmetic of $\cS_3 \subseteq \cA_3$. Below, we prove this result using the preceding proposition in a similar way that Theorem \ref{thm:genus4} was used in the proof of Corollary~\ref{cor:main_oort_a4}. 

\begin{comment}
    If $g = 2$, then every component of $\cS_2$ is $1$-dimensional and its generic point corresponds to a Jacobian of a smooth curve, which is necessarily hyperelliptic. Torelli's theorem \cite[Theorem 12.1]{milne} and Proposition \ref{prop:genus2} complete the proof in this case.
\end{comment}

\begin{cor}
 Oort's conjecture \ref{conj:oort} holds for $g = 3$ and any $p>2$.
\end{cor}
\begin{proof}
 By \cite[Theorem 1.12]{oort_hess}, we know that any component of $\cS_3$ is $2$-dimensional and that its generic point corresponds to a Jacobian of a smooth curve, which is non-hyperelliptic. Similarly as in the proof of Corollary \ref{cor:main_oort_a4}, it is enough to show that the generic point $C$ of any component $\Gamma$ of $\cM_3^{ss}$ has a trivial automorphism group. This follows from Proposition~\ref{prop:genus3}. 
\end{proof}

We end this section by presenting a new short proof that we could not find in the literature of a famous result by Oort in \cite[Theorem 5.6]{oort_hess}, stating that there are no smooth supersingular hyperelliptic curves of genus $3$ in characteristic $2$. Our proof uses an interplay between the Newton polygon and Ekedahl-Oort strata as well as some restrictions on possible Ekedahl-Oort types of hyperelliptic curves specific to characteristic $2$. 

\begin{prop} In characteristic $2$, there does not exist a smooth supersingular hyperelliptic curve of genus $3$.
\end{prop}
\begin{proof} 
 By Oort's minimality \cite[Theorem 1.2]{oort_minimal}, for every $g\geq 2$ and every Newton polygon $\xi$ occurring for some $g$-dimensional p.p.a.v. there is a Young diagram $\mu_{\xi}$ such that $$Z_{\mu_{\xi}} \subseteq \mathcal{N}_{\xi},$$ where $\mathcal{N}_{\xi}\subseteq \cA_g$ is the locus of all p.p.a.v.'s of dimension $g$ whose Newton polygon equals~$\xi$. From Harashita's computations in \cite{harashita}, if $\xi = \xi_{\frac{1}{3}}$ is the Newton polygon for $g = 3$ with slopes $\frac{1}{3}$ and $\frac{2}{3}$, it follows that $\mu_{\xi} = [3, 1]$. In other words, every $3$-dimensional p.p.a.v. whose Ekedahl-Oort type equals $[3, 1]$ has Newton polygon $\xi_{\frac{1}{3}}$. 
 
 By \cite[Theorem~3.2]{vdgeercycle}  or \cite[Corollary 5.3]{elkinpries}, for every smooth hyperelliptic curve~$C$ of genus $3$ with $2$-rank $0$ it holds that $\mu(\cJ_{C}) = [3, 1]$. Therefore, the Newton polygon of $\cJ_C$ equals $\xi_{\frac{1}{3}}$ so that $C$ cannot be supersingular. In other words, $\cH_3^{ss}$ is empty in characteristic~$2$.
\end{proof}

\end{document}